\documentclass[11pt,leqno]{article}%
\usepackage[margin=15mm]{geometry}
\usepackage{amsmath}
\usepackage{graphicx}
\usepackage{url}
\usepackage{amsfonts}
\usepackage{amssymb}
\usepackage{amsthm}

\usepackage{epsfig}

\usepackage[all]{xypic}

\newtheorem{theorem}{Theorem}[section]

\newtheorem{proposition}[theorem]{Proposition}
\newtheorem{remark}[theorem]{Remark}

\begin{document}

\title{Hydrodynamic--type systems describing $2$-dimensional polynomially integrable geodesic flows}
%
%
 \author{Gianni Manno\footnote{Dipartimento di Scienze Matematiche ``G. Lagrange'', Politecnico di Torino, Corso Duca degli Abruzzi, 24, 10129 Torino, ITALY. giovanni.manno@polito.it}, Maxim V. Pavlov\footnote{
  Sector of Mathematical Physics, Lebedev Physical Institute of Russian Academy of Sciences,
   Leninskij Prospekt 53, 119991 Moscow, Russia, 
  Department of Applied Mathematics,National Research Nuclear University MEPHI,Kashirskoe Shosse 31, 115409 Moscow, Russia,
  Department of Mechanics and Mathematics,Novosibirsk State University, 2 Pirogova street, 630090, Novosibirsk, Russia, mpavlov@itp.ac.ru}}

\maketitle

\begin{abstract}
Starting from a homogeneous polynomial in momenta of arbitrary order we extract multi-component
hydrodynamic-type systems which describe $2$-dimensional geodesic flows admitting the initial polynomial as integral. All these
hydrodynamic-type systems are semi-Hamiltonian, thus implying that they are
integrable according to the generalized hodograph method. Moreover, they are
integrable in a constructive sense as polynomial first integrals allow to
construct generating equations of conservation laws. According to the multiplicity of the roots of the polynomial integral, we separate integrable particular cases.
\end{abstract}

\textbf{Keyword:}
Integrable geodesic flows, semi-Hamiltonian hydrodynamic systems


\textbf{MSC 2010:} 53D25, 37J35, 37K05, 37D40, 70H05



\section*{Introduction}

This paper is devoted to the following classical problem: how to extract
geodesic flows that are integrable according to the Liouville theorem. As
usual, the geodesic flow of an $n$-dimensional (pseudo-)Riemannian manifold
$(M,g)$ is locally described by a system of (generally nonlinear) ODEs
\begin{equation}
\label{eq.geod}\ddot{x}^{i}+\Gamma_{jk}^{i}\dot{x}^{j}\dot{x}^{k}=0\,, \quad
i=1,\dots,n\,,
\end{equation}
where $\mathbf{x}=(x^{i})$ is a system of coordinates of $M$, $\Gamma_{jk}^{i}(\mathbf{x})$ are the Christoffel symbols of the Levi-Civita connection
and $\dot{\mathbf{x}}$, $\ddot{\mathbf{x}}$ are, respectively, the first and
second derivatives of $\mathbf{x}$ w.r.t. an external parameter $t$. We denote
by $(x^{i},p_{i})$ the system of coordinates of $T^{*}M$ induced by
coordinates $(x^{i})$. It is well known that system \eqref{eq.geod} can be
written in Hamiltonian form
\begin{equation*}
\dot{x}^{i}=\frac{\partial H}{\partial p_{i}},\text{ \ }\dot{p}_{i}
=-\frac{\partial H}{\partial x^{i}},\ \ i=1,\ldots,n,\label{1}
\end{equation*}
where
\begin{equation}\label{eq.hamil}
H(\mathbf{x},\mathbf{p})=\frac{1}{2}g^{km}(\mathbf{x})p_{k}p_{m}
\end{equation}
and $\mathbf{p}=(p_{i})$. Since the Hamiltonian function $H(\mathbf{x},\mathbf{p})$ does not depend on $t$ explicitly, Hamilton's equations are
Liouville integrable if there exist $n-1$ first integrals $f_{k}(\mathbf{x},\mathbf{p})$ in involution, i.e. $\{f_{k},H\}=0$ and $\{f_{i},f_{k}\}=0$, where
\[
\{f,g\}=\frac{\partial f}{\partial x^{i}}\frac{\partial g}{\partial p_{i}}-\frac{\partial f}{\partial p_{i}}\frac{\partial g}{\partial x^{i}}
\]
is the usual Poisson bracket.

In the present paper we shall consider only the case $n=2$. In such a case a first
integral $f(x^{1},x^{2},p_{1},p_{2})$ satisfies the equation
\begin{equation}\label{eq.fH}
\{f,H\}=\frac{\partial f}{\partial x^{1}}\frac{\partial
H}{\partial p_{1}}-\frac{\partial f}{\partial p_{1}}\frac{\partial H}{\partial
x^{1}}+\frac{\partial f}{\partial x^{2}}\frac{\partial H}{\partial p_{2}}-\frac{\partial f}{\partial p_{2}}\frac{\partial H}{\partial x^{2}}=0.
\end{equation}

A central idea is the research of first integrals $f(x^{1},x^{2},p_{1},p_{2})$ that are homogeneous polynomials of degree $N$ in momenta $p_{i}$, i.e. of the form
\begin{equation}\label{eq.f}
f=\overset{N}{\sum_{m=0}}a_{m}(x^{1},x^{2})p_{1}^{N-m}p_{2}^{m}\,.
\end{equation}
By substituting the polynomial
ansatz \eqref{eq.f} into \eqref{eq.fH} one can derive a quasi-linear system of
first order PDEs (see a similar approach, for instance, in \cite{Bialy.Mironov}).

The problem of describing $2$-dimensional metrics admitting a polynomial integral is classical: its formulation is due at least to Darboux \cite{Da} and
it has been solved (both locally and
globally) in the case when $N$ is equal either to $1$ or $2$. For $N=2$ see
\cite{Krug}, where metrics admitting quadratic integrals are called of
\emph{Liouville} type: in that paper the author solved the problem posed in
\cite{vagner} of their local characterization (see also \cite{BMM} for the
relationship with superintegrable metrics -in the sense of \cite{KKM}- and
projectively equivalent metrics). For $N=3$ see \cite{inv,super}, where the
problem of finding such metrics is solved under particular assumptions.

A possible strategy (see, for instance, \cite{Bialy}) is to fix a coordinate system where the metric assumes some
special form, for instance: $ds^{2}=a(x^{1},x^{2})(dx^{1})^{2}+(dx^{2})^{2}$ (semi-geodesic coordinates),  $ds^{2}=a(x^{1},x^{2})[(dx^{1})^{2}+(dx^{2})^{2}]$ (isothermal coordinates),
$ds^{2}=(dx^{1})^{2}+a(x^{1},x^{2})dx^{1}dx^{2}+(dx^{2})^{2}$ (Chebyshev coordinates), etc. In all these cases condition \eqref{eq.fH}, with $f$ given by \eqref{eq.f}, leads to hydrodynamic-type systems of PDEs on the coefficients
$a_{m}(x^{1},x^{2})$. Integrability of such
hydrodynamic-type systems is a separate question.

In this paper we present an alternative construction. We suppose that
polynomial \eqref{eq.f} has $N$ real roots (not necessarily distinct), so that
\eqref{eq.f} can be written in the following factorized form
\begin{equation}\label{eq.pol.iniziale}
f=\underset{m=1}{\overset{N}{{\displaystyle\prod} }}(\alpha_{m}^{1}p_{1}+\alpha_{m}^{2}p_{2})\, .
\end{equation}
By finding a suitable system of coordinates where both \eqref{eq.pol.iniziale} and Hamiltonian \eqref{eq.hamil} assume a particular convenient form, by means of condition \eqref{eq.fH}, we arrive to discuss a first-order quasi-linear system of PDEs in the unknown functions $\alpha_{m}^{1}(x^{1},x^{2}),\alpha_{m}^{2}(x^{1},x^{2})$ and $g_{ik}(x^{1},x^{2})$.
Our main achievement is a more general ansatz for polynomial
integrals than that presented in \cite{Bialy}: the semi-Hamiltonian
hydrodynamic-type system we obtained contains $N+1$ equations, while the hydrodynamic-type
system considered in \cite{Bialy} contains just $N$ equations. Furthermore, the system we obtained possesses a simple reduction (i.e. the metric is a linear
expression in terms of field variables) to the case considered in
\cite{Bialy}. In particular, one can select first order quasi-linear systems according to the multiplicity of roots of polynomial \eqref{eq.pol.iniziale}.

\subsection*{Notations and conventions}

$(M,g)$ will be always a $2$-dimensional (pseudo-)Riemannian manifold. The
symmetric tensor product is denoted by $\odot$.
We shall denote by $f_x$ the derivative of $f$ w.r.t. $x$. All results presented in the
paper are of local character, meaning that we always work in suitable neighborhoods.

\section{Polynomial integrals of the geodesic flow}\label{sec.1}

For our purposes we need the following proposition, that is well-known.
\begin{proposition}\label{prop.p1.3}
If the geodesic flow of $(M,g)$ admits a polynomial first integral $f$ of the following form
\begin{equation*}
f=\mp(\alpha^1p_1+\alpha^2p_2)^N
\end{equation*}
then it admits also the first linear integral $\alpha^1p_1+\alpha^2p_2$. In particular, $(M,g)$ admits a Killing vector field.
\end{proposition}


%
%
Essentially, the case described in Proposition \ref{prop.p1.3} is the most
simple one that can occur, i.e. when the metric admits a local infinitesimal
isometry. Thus, in what follows, we shall focus our attention to the case when the geodesic flow of $(M,g)$ admits a first integral $f$ with at least two distinct roots. We shall use the results of next Proposition to get a description of a such
geodesic flow in terms of a hydrodynamic-type system.
\begin{proposition}\label{prop.normal.forms.pol}
Let $f$ be as in \eqref{eq.pol.iniziale}. Let us
suppose that $f$ admits at least two distinct roots, i.e.
\begin{equation}\label{eq.appoggio.2}
f=(\alpha^1_1p_1+\alpha^2_1p_2)(\alpha^1_2p_1+\alpha^2_2p_2)Q
\end{equation}
where $Q$ is a homogeneous polynomial of degree $N-2$ in momenta, $\alpha^1_1(p)\alpha^2_2(p)-\alpha^1_2(p)\alpha^2_1(p)\neq 0$, $(\alpha^1_1(p),\alpha^2_1(p))\neq 0\neq(\alpha^1_2(p),\alpha^2_2(p))$ at a point $p\in M$. Then, in a neighborhood $U$ of the point $p$, there
exists a system of coordinates $(x_{\mathrm{new}}^{1},x_{\mathrm{new}}^{2})$
such that polynomial $f$, in the induced system of coordinates
$(x_{\mathrm{new}}^{1},x_{\mathrm{new}}^{2},{p_{1}}_{{\mathrm{new}}},{p_{2}}_{{\mathrm{new}}})$ of $T^{\ast}M$, assumes the form
\begin{equation}
f=p_{1}p_{2}\prod_{i=1}^{N-2}(\alpha_{i}^{1}p_{1}+\alpha_{i}^{2}p_{2})\,,\quad\alpha_{j}^{k}=\alpha_{j}^{k}(x^{1},x^{2})
\label{eq.pol.2.distinct.roots.general}
\end{equation}
\end{proposition}
\begin{proof}
In view of the identification $p_i\simeq\partial_{x^i}$, we can write polynomial \eqref{eq.appoggio.2} as follows:
\begin{equation*}
f=X\odot Y\odot \Xi
\end{equation*}
where $X=\alpha^1_1\partial_{x^1}+\alpha^2_1\partial_{x^2}$, $Y=\alpha^1_2\partial_{x^1}+\alpha^2_2\partial_{x^2}$ and $\Xi$ is a $(N-2,0)$-tensor. The two distributions of curves, formed by the integral curves of $X$ and $Y$, can always be chosen as coordinate lines: in these coordinates the integral acquires a factor of $p_1p_2$, i.e. polynomial \eqref{eq.appoggio.2} assumes the form \eqref{eq.pol.2.distinct.roots.general}.
%
%
%
\end{proof}

\begin{remark}\label{rem.extra.freedom}
The form of any homogeneous polynomial in momenta
(in particular polynomial \eqref{eq.pol.2.distinct.roots.general}) does not
change under a transformation
\begin{equation}
x_{\mathrm{new}}^{1}=x_{\mathrm{new}}^{1}(x^{1})\,,\quad x_{\mathrm{new}}^{2}=x_{\mathrm{new}}^{2}(x^{2})\,.\label{eq.x.new}
\end{equation}
Indeed, both $\partial_{x^{1}}$ and $\partial_{x^{2}}$ do not change direction
under the above transformation.
\end{remark}

\begin{proposition}
\label{prop.ham.and.pol} If the geodesic flow of $(M,g)$ admits a
polynomial integral $f$ with the properties described in Proposition \ref{prop.normal.forms.pol}, then there exists a
system of coordinates $(x^{i},p_{i})$ where the Hamiltonian \eqref{eq.hamil}
has the form
\begin{equation}
\label{eq.H.normal.form}H=\frac{1}{2}\epsilon_{1}p_{1}^{2}+g^{12}p_{1}p_{2}+\frac{1}{2}\epsilon_{2}p_{2}^{2}\,,\quad g^{12}=g^{12}(x^{1},x^{2})\,,\quad\epsilon_{i}\in\{-1,0,1\}\,.
\end{equation}
and the polynomial $f$ the form \eqref{eq.pol.2.distinct.roots.general}.
\end{proposition}

\begin{proof}
Let $f$ be a polynomial integral of $(M,g)$ with two distinct roots. Then, in view of Proposition \ref{prop.normal.forms.pol}, there exists a system of coordinates where $f$ has the form \eqref{eq.pol.2.distinct.roots.general}. Equivalently, it is the same of considering form \eqref{eq.f} with $a_0=0=a_N$. We shall use this latter notation. If we substitute this $f$ in \eqref{eq.fH} we obtain
$$
a_{1}g^{11}_{x^2}=0\,, \quad a_{N-1}g^{22}_{x^1}=0
$$
as they are, respectively, the coefficients of $p_1^{N+1}$ and $p_2^{N+1}$ of left hand side term of \eqref{eq.fH}. If both $a_{1}$ and $a_{N-1}$ are not zero, then $g^{11}_{x^2}=0=g^{22}_{x^1}$, implying that $g^{11}=g^{11}(x^1)$ and $g^{22}=g^{22}(x^2)$. If $a_{1}=0$, then we obtain $a_{2}g^{11}_{x^2}=0$ as $a_{2}g^{11}_{x^2}$ is the coefficient of $p_1^Np_2$ of the left hand side of \eqref{eq.fH}, that gives either $g^{11}_{x^2}=0$ or $a_{2}=0$. If the latter case occurs, then we obtain $a_{3}g^{11}_{x^2}=0$, and so further. A similar reasoning applies, of course, also if we start from $a_{N-1}g^{22}_{x^1}$. We conclude that if the initial polynomial integral is not zero, then $g^{11}_{x^2}=0=g^{22}_{x^1}$, so that
\begin{equation}\label{eq.g.meno.1}
g^{-1}=g^{11}(x^1)\partial_{x^1}\odot\partial_{x^1}+2g^{12}(x^1,x^2)\partial_{x^1}\odot
\partial_{x^2}+g^{22}(x^2)\partial_{x^2}\odot\partial_{x^2}\,.
\end{equation}
Now, in view of Remark \ref{rem.extra.freedom}, we can use a changing of coordinates \eqref{eq.x.new} to further simplify \eqref{eq.g.meno.1}. In particular, we can find a new system of coordinates such that $g^{11}=\epsilon_1$ and $g^{22}=\epsilon_2$, with $\epsilon_{i}\in\{-1,0,1\}$. Thus, in these new system of coordinates, Hamiltonian $H$ (see \eqref{eq.hamil}) assumes the form \eqref{eq.H.normal.form} (up to renaming function $g^{12}$).
\end{proof}
In our case the metric with upper indices looks precisely like the metric with lower indices in the case of the Chebyshev coordinate net, see formula \eqref{eq.metric.nonso} below. Thus, our coordinate system is conformally equivalent to the Chebyshev net. However, our choice of coordinates is very convenient for our further computations. For instance we shall show in Section \ref{sec.pol.int.3.4} that in this coordinate system all further reductions appear in the most simple way.

\smallskip
To not overload the notation, from now on we shall consider only the
Hamiltonian
\begin{equation}\label{eq.ham.we.considered}
H=\frac{1}{2}p_{1}^{2}+g^{12}p_{1}p_{2}+\frac
{1}{2}p_{2}^{2}\,,
\end{equation}
i.e. \eqref{eq.H.normal.form} with $\epsilon_{1}=\epsilon_{2}=1$, as the other
cases can be treated in the same way. As a possible polynomial integral of the
above Hamiltonian, without loss of generality, in view of
Proposition \ref{prop.ham.and.pol}, we can consider
\begin{equation}
f=a_{1}p_{1}^{N-1}p_{2}+a_{2}p_{1}^{N-2}p_{2}^{2}+...+a_{N-2}p_{1}^{2}p_{2}^{N-2}+a_{N-1}p_{1}p_{2}^{N-1}. \label{Polynomial}
\end{equation}
Then substituting (\ref{Polynomial}) into \eqref{eq.fH} we obtain
\begin{equation}
(p_{1}+g^{12}p_{2})f_{x^1}-p_{1}p_{2}(g^{12})_{x^{1}}f_{p_1}+(g^{12}p_{1}+p_{2})f_{x^2} -p_{1}p_{2}(g^{12})_{x^{2}}f_{p_2}=0 \label{involution}
\end{equation}
from which one can derive a quasi-linear system of first order PDEs
\begin{equation}
\label{raz}\left\{
\begin{array}
[c]{l}
a_{1,x^{1}}+g^{12}a_{1,x^{2}}=a_{1}(g^{12})_{x^{2}}\\
\\
a_{k,x^{1}}+g^{12}a_{k-1,x^{1}}+g^{12}a_{k,x^{2}}+a_{k-1,x^{2}}=ka_{k}
(g^{12})_{x^{2}}+(N+1-k)a_{k-1}(g^{12})_{x^{1}},\text{ \ }k=2,...,N-1\\
\\
g^{12}a_{N-1,x^{1}}+a_{N-1,x^{2}}=a_{N-1}(g^{12})_{x^{1}}
\end{array}
\right.
\end{equation}
with $N$ unknown functions, i.e. the metric coefficient $g^{12}$ and the $N-1$
coefficients $a_{k}$ of polynomial ansatz \eqref{Polynomial}.

A very important property of this system is existence of simple reductions
based on the multiplicity of the roots of polynomial \eqref{Polynomial}. For
instance: $a_{k}=0,k=1,2,...,K_{1}<N-1$ and $k=N-1,N-2,...,K_{2}<N-1$. As an example, if $N=5$, we have the following full list of distinguish reductions: $a_{1}
=0$; $a_{1}=a_{2}=0$; $a_{1}=a_{2}=a_{3}=0$; $a_{1}=a_{4}=0$; $a_{1}
=a_{2}=a_{4}=0$. Thus, instead to investigate each particular case, one
can concentrate on the generic system \eqref{raz} only.

However this system is written in a non-evolutionary form. By this reason, in
Section \ref{sec.integ}, we rewrite it in an evolutionary form by an
appropriate reciprocal transformation.

\subsection{Polynomial integrals of third and fourth degree}\label{sec.pol.int.3.4}

Here we consider system \eqref{raz} in the case of polynomial integrals of third and fourth degree. We also briefly discuss  its reductions according to the multiplicity of roots of the polynomial integral.

\subsubsection{Geodesic flows admitting homogeneous polynomial integrals of
third and fourth degree in momenta}

\label{sec.third}

Let $f$ be a homogeneous polynomial of a third degree in momenta. We already
seen that, if $f$ is a perfect cube, the metric admitting such $f$ as an
integral it admits a Killing vector field (see Proposition \ref{prop.p1.3}).
So, let us then consider the case when polynomial \eqref{eq.pol.iniziale} has
at least two distinct roots. In this case, in view of Proposition
\ref{prop.normal.forms.pol}, a normal form of such polynomial is
\eqref{Polynomial} with $N=3$, i.e.
\begin{equation}
\label{eq.pol.this.case}a_{1}p_{1}^{2}p_{2}+a_{2}p_{1}p_{2}^{2}\,,\quad
a_{1},a_{2}\in C^{\infty}(M).
\end{equation}
that, in view of system \eqref{raz}, is an integral of Hamiltonian
\eqref{eq.ham.we.considered} iff
\begin{equation}
\left\{
\begin{array}
[c]{l}
a_{1,x^{1}}+a_{1,x^{2}}g^{12}=(g^{12})_{x^{2}} a_{1}\\
\\
\displaystyle{a_{1,x^{1}}g^{12}+a_{2,x^{1}}+a_{1,x^{2}} +a_{2,x^{2}}g^{12}
=2(g^{12})_{x^{1}} a_{1}}+2(g^{12})_{x^{2}} a_{2}\\
\\
a_{2,x^{1}}g^{12}+a_{2,x^{2}} =(g^{12})_{x^{1}} a_{2}
\end{array}
\right.  \label{eq.sys.3.degree.general.bis}
\end{equation}
If polynomial $f$ admits two coincident roots, then its normal form is
\eqref{eq.pol.this.case} with $a_{1}=0$ (or, equivalently, $a_{2}=0$) and,
correspondingly, the quasi-linear system to be considered is
\eqref{eq.sys.3.degree.general.bis} with either $a_{1}=0$ or $a_{2}=0$.

\subsubsection{Geodesic flows admitting homogeneous polynomial integrals of
fourth degree in momenta}

Let $f$ be a homogeneous polynomial of fourth degree in momenta. Let us
suppose that $f$ admits at least two distinct roots, otherwise $(M,g)$ admits
a Killing vector field (see again Proposition \ref{prop.p1.3}). In this case,
in view of Proposition \ref{prop.normal.forms.pol}, a normal form of such
polynomial is \eqref{Polynomial} with $N=4$, i.e.
\begin{equation}\label{eq.pol.this.case.fourth}
a_{1}p_{1}^{3}p_{2}+a_{2}p_{1}^{2}p_{2}
^{2}+a_{3}p_{1}p_{2}^{3}\,,\quad a_{1},a_{2},a_{3}\in C^{\infty}(M).
\end{equation}
that, in view of system \eqref{raz}, is an integral of Hamiltonian
\eqref{eq.ham.we.considered} iff
\begin{equation}
\label{eq.4deg.4diff.bis}\left\{
\begin{array}
[c]{l}
a_{1,x^{1}} + a_{1,x^{2}}g^{12} = a_{1} (g^{12})_{x^{2}}\\
\\
a_{1,x^{2}} + a_{2,x^{1}} + a_{1,x^{1}}g^{12} + a_{2,x^{2}}g^{12} = 2 a_{2}
(g^{12})_{x^{2}} + 3 a_{1}(g^{12})_{x^{1}}\\
\\
a_{3,x^{1}} + a_{2,x^{2}} + a_{3,x^{2}}g^{12} + a_{2,x^{1}}g^{12} = 3
a_{3}(g^{12})_{x^{2}} +2 a_{2}(g^{12})_{x^{1}}\\
\\
a_{3,x^{2}} + a_{3,x^{1}}g^{12} = a_{3}(g^{12})_{x^{1}}
\end{array}
\right.
\end{equation}

According to the multiplicity of roots of polynomial
\eqref{eq.pol.this.case.fourth}, we have several possible reductions of system
\eqref{eq.4deg.4diff.bis}. In fact, by arguing as in the end of Section
\ref{sec.third}, if in system \eqref{eq.4deg.4diff.bis} we put $a_{2}
=a_{3}=0$, then we obtain the system describing metrics admitting a polynomial
integral with three coincident roots. If we put $a_{1}=a_{3}=0$, we obtain the
system describing metrics admitting a polynomial integral with two coincident
roots. Finally, if we put $a_{3}=0$, we obtain the system describing metrics
admitting a polynomial integral with two coincident roots (and two distinct).
If no particular assumption is imposed on $a_{1},a_{2},a_{3}$, then,
generically, the above system describes metrics admitting polynomial integrals
with all distinct roots.

\section{Integrability}\label{sec.integ}

In this section we are going to investigate integrability of hydrodynamic-type system \eqref{raz} selected by the Hamiltonian \eqref{eq.ham.we.considered} and the polynomial ansatz \eqref{Polynomial}. For our further research we need first to reduce this system to an evolutionary form. This is possible by finding an appropriate reciprocal transformation; below we present a constructive algorithm:
\begin{itemize}
\item We derive the Liouville equation from commutativity of the Hamiltonian and the first integral \eqref{eq.fH};
\item we introduce an appropriate reciprocal transformation (to semi--geodesic coordinates);
\item under this transformation we recompute the metric, the Hamilton-Jacobi equation, the Liouville equation, momenta and finally the hydrodynamic--type system \eqref{raz};
\item we discuss the existence of infinite set of conservation laws;
\item we prove the diagonalizability of this hydrodynamic--type system;
\item as an example, we shall consider the two component case.
\end{itemize}

\medskip
Polynomial ansatz \eqref{Polynomial} can be written in the form (see more details in
\cite{classmech,MaxTsarMB})
\begin{equation}
f=\left(  p_{1}^{2}+2g^{12}p_{1}p_{2}+p_{2}^{2}\right)  ^{N/2} \lambda(s,x^{1},x^{2}),\label{ff}
\end{equation}
where $s=p_{2}/p_{1}$ and the function $\lambda(s,x^{1},x^{2})$ satisfies \eqref{involution}:
\begin{equation}
(1+g^{12}s)\lambda_{x^1}+(g^{12}+s)\lambda_{x^2}+\left(s^{2}(g^{12})_{x^{1}}-s(g^{12})_{x^{2}}\right)\lambda_{s}  =0.\label{lin}
\end{equation}
Introducing the variable
\begin{equation}\label{ps}
p=(1+2g^{12}s+s^{2})^{-1/2}
\end{equation}
(instead of the variable $s$), equation \eqref{lin} reduces to the canonical form\footnote{Classifications of such Hamiltonian equations is given in \cite{OdPavSok}.}
\begin{equation}\label{lambda}
\lambda_{x^{2}}=\{\lambda,H\}=H_{p}\lambda_{x^{1}}-\lambda_{p}H_{x^{1}
},
\end{equation}
where now the function $\lambda$ depends on $p$ via formula \eqref{ps} and the Hamiltonian function is
\begin{equation}\label{Jacobi}
H=\sqrt{\left(  (g^{12})^{2}-1\right)  p^{2}+1}-g^{12}p.
\end{equation}

Equation (\ref{lambda}) plays an important role in the theory of integrable
hydrodynamic chains and semi-Hamiltonian hydrodynamic-type systems (see, for
instance, \cite{maksgen,algebra}). The existence  of representation
\eqref{lambda} for hydrodynamic-type system \eqref{raz} means that such a system
is integrable (or semi-Hamiltonian), i.e. it admits infinitely
many conservation laws, commuting flows and particular solutions (see more details in \cite{Tsar,tsar91}).
The integrability procedure means that instead of the function $\lambda(x^1,x^2,p)$ we consider the function $\lambda(p,g^{12},a_1(x^1,x^2),\dots,a_{N-1}(x^1,x^2))$ whose existence is equivalent to the integrability of some overdetermined system (now known as the Gibbons-Tsarev system \cite{GibbonsTsarev}). Substitution of one of its particular solutions (see \eqref{Polynomial})
\begin{multline*}
\lambda(s,x^1,x^2)=\left(  p_{1}^{2}+2g^{12}p_{1}p_{2}+p_{2}^{2}\right)^{-N/2} f(x^1,x^2,p_1,p_2)
\\
=(p_{1}^{2}+2g^{12}p_{1}p_{2}+p_{2}^{2})^{-N/2}(a_{1}p_{1}^{N-1}p_{2}+a_{2}p_{1}^{N-2}p_{2}^{2}+...+a_{N-2}p_{1}^{2}p_{2}^{N-2}+a_{N-1}p_{1}p_{2}^{N-1})
\\
=(1+2g^{12}s+s^2)^{-N/2}(a_{1}s+a_{2}s^2+...+a_{N-2}s^{N-2}+a_{N-1}s^{N-1})
\end{multline*}
into \eqref{lambda} creates the hydrodynamic--type system \eqref{raz}. The Liouville--type equation \eqref{lambda} takes the form a Hamilton--Jacobi equation
\begin{equation}\label{eq.px2.hx1}
p_{x^2}=H_{x^1}
\end{equation}
where differentiation of $H$ (given by \eqref{Jacobi}) w.r.t. $x^1$ means to differentiate not only $g^{12}$ but also $p$, that now is a dependent function on $(x^1,x^2,\lambda)$. So, now the function $p$ is a generating function of conservation law densities for hydrodynamic--type system \eqref{raz}, while the Hamilton-Jacobi equation \eqref{eq.px2.hx1} plays the role of the generating equation of the corresponding conservation laws (w.r.t. the parameter $\lambda$).

\subsection{Transformation to the Semi-Geodesic Coordinates}

In this section we are going to consider the main system of our interest \eqref{raz} in another coordinate system that is more convenient for our further computations. Instead of our original coordinates (conformally equivalent Chebyshev coordinates, see the discussion before formula \eqref{eq.ham.we.considered}), we introduce the so called \emph{semi--geodesic} coordinates $(x,y)$  (see more details in \cite{Bialy.Mironov}), which we determine by virtue of the reciprocal transformation
\begin{equation}\label{recip}
dy=\frac{1}{a_{N-1}}dx^{1}-\frac{g^{12}}{a_{N-1}}dx^{2}\,,\quad
dx=dx^{2}\,,
\end{equation}
where the potential function $y$ follows from
the third equation of system \eqref{raz} written in the conservative
form
\[
\left(  \frac{1}{a_{N-1}}\right)_{x^{2}}+\left(  \frac{g^{12}}{a_{N-1}}\right)_{x^{1}}=0\,,
\]
i.e.
$$
y_{x^1}=\frac{1}{a_{N-1}}\,,\quad y_{x^2}=-\frac{g^{12}}{a_{N-1}}\,.
$$
So, the metric corresponding to the Hamiltonian \eqref{eq.ham.we.considered}
\begin{equation}\label{eq.metric.nonso}
ds^{2}=\frac{1}{1-(g^{12})^{2}}\left((dx^{1})^{2}-2g^{12}dx^{1}dx^{2}+(dx^{2})^{2}\right)\,,
\end{equation}
in these semi--geodesic coordinates $(x,y)$, assumes the form
\begin{equation}\label{eq.metric.nonso.2}
ds^{2}=(dx)^{2}+\frac{a_{N-1}^{2}}{1-(g^{12})^{2}}(dy)^{2}\,.
\end{equation}
Correspondingly, the Hamilton--Jacobi equation (see \eqref{eq.px2.hx1})
\begin{equation}\label{eq.px2.bla}
p_{x^{2}}=\left(  \sqrt{((g^{12})^{2}-1)p^{2}+1}-g^{12}p\right)_{x^{1}}
\end{equation}
becomes (see again \cite{Bialy.Mironov} and other details in \cite{MaxTsarMB})
\begin{equation}\label{hj}
\tilde{p}_{y}=\left(  \frac{a_{N-1}}{\sqrt{1-(g^{12})^{2}}}\sqrt{1-\tilde
{p}^{2}}\right)_{x},
\end{equation}
where we define $\tilde{p}$ as follows
\begin{equation}\label{eq.p.tilde}
\tilde{p}=\sqrt{((g^{12})^{2}-1)p^{2}+1}\,
\end{equation}
(we remind that the variable $p$ has been defined by \eqref{ps}). In fact, the conservation law \eqref{eq.px2.bla} can be written in the potential form
$$
d\xi=pdx^1+\left( \sqrt{((g^{12})^{2}-1)p^{2}+1}-g^{12}p \right)dx^2\,.
$$
Under the inverse reciprocal transformation (see \eqref{recip})
\begin{equation}\label{eq.recip.non.so}
dx^1=g^{12}dx+ a_{N-1}dy \,,\quad dx^2=dx
\end{equation}
we obtain
$$
d\xi=pa_{N-1}dy+\sqrt{((g^{12})^2-1)p^2+1}\,dx\,.
$$
The compatibility condition $(\xi_x)_y=(\xi_y)_x$ leads to \eqref{hj} where  $\tilde{p}$ is given by \eqref{eq.p.tilde}.
To recompute first integral \eqref{Polynomial} via new coordinates $(x,y)$ we need first to recompute the corresponding momenta. To do this
we use the identity
\[
p_{1}dx^{1}+p_{2}dx^{2}=\tilde{p}_{1}dx+\tilde{p}_{2}dy\,.
\]
In this case
\begin{equation}\label{eq.p1.p2.tilde}
\tilde{p}_{1}=p_{2}+g^{12}p_{1},\text{ \ }\tilde{p}_{2}=a_{N-1}p_{1}\,,
\end{equation}
then the first integral \eqref{Polynomial} takes the form
\begin{equation}\label{polynom}
f=\tilde{p}_{2}(\tilde{p}_{1})^{N-1}+\tilde{a}_{1}(\tilde{p}_{2})^{2}
(\tilde{p}_{1})^{N-2}+\tilde{a}_{2}(\tilde{p}_{2})^{3}(\tilde{p}_{1}
)^{N-3}+...+\tilde{a}_{N-2}(\tilde{p}_{2})^{N-1}\tilde{p}_{1}+\tilde{a}
_{N-1}(\tilde{p}_{2})^{N}.
\end{equation}

Note that all coefficients $\tilde{a}_{k}$ are linear functions with respect to
$a_{m}$ and polynomial functions with respect to $g^{12}$ and $(a_{N-1})^{-1}$.
Taking into account that first integral \eqref{polynom} can be written in the factorized form
\[
f=(\tilde{p}_{2})^{N}\overset{N-1}{\underset{m=1}{
{\displaystyle\prod}
}}\left(  \frac{\tilde{p}_{1}}{\tilde{p}_{2}}-\tilde{b}_{m}\right)  ,
\]
we introduce (cf. \eqref{ff}) the function
\[
\tilde{\lambda}(\tilde{s},x,y)=\left(  \frac{1}{\tilde{s}^{2}}+\frac{1-(g^{12})^{2}
}{a_{N-1}^{2}}\right)  ^{-N/2}\overset{N-1}{\underset{m=1}{
{\displaystyle\prod}
}}\left(  \frac{1}{\tilde{s}}-\tilde{b}_{m}\right)  ,
\]
where $\tilde{s}=\tilde{p}_{2}/\tilde{p}_{1}$. For our further convenience we define an appropriate independent variable $q$ instead of $\tilde{s}$:
\begin{equation}\label{q.non.so}
q=\frac{a_{N-1}}{\sqrt{1-(g^{12})^{2}}}\frac{1}{\tilde{s}}.
\end{equation}
Then we obtain the equation of the Riemann surface
\begin{equation}\label{lyambda}
\tilde{\lambda}(x,y,q)=a^{-1/2}(1+q^{2})^{-N/2}\overset{N-1}{\underset{m=1}{
{\displaystyle\prod}
}}(q-b_{m}),
\end{equation}
where $\tilde{b}_{k}=a^{1/2}b_{k}$ and
\[
a^{-1/2}=\frac{a_{N-1}}{\sqrt{1-(g^{12})^{2}}}.
\]
Under the reciprocal transformation \eqref{recip} Liouville equation \eqref{lambda} takes the form
\begin{equation}
\tilde{\lambda}_{y}=a^{-1/2}q\tilde{\lambda}_{x}+(1+q^{2})\tilde{\lambda}
_{q}(a^{-1/2})_{x}\label{liouville}
\end{equation}
where (see \eqref{eq.p1.p2.tilde} and \eqref{q.non.so})
\begin{equation}\label{eq.p.q}
\tilde{p}=\frac{q}{1+q^2}\,,\quad q=\frac{\tilde{p}}{1-\tilde{p}^2}.
\end{equation}
The Liouville--type equation \eqref{liouville} takes again the form of  Hamilton--Jacobi equation \eqref{hj} (see \eqref{eq.px2.hx1}), that was already investigated in
\cite{MaxTsarMB}.
Substitution
\eqref{lyambda} into \eqref{liouville} yields the hydrodynamic-type
system\footnote{another hydrodynamic-type system was found in
\cite{Bialy.Mironov}. The approach presented in this Section was established
in \cite{algebra} and utilized in \cite{classmech} for the hydrodynamic-type
system derived by V.V. Kozlov in \cite{kozlov}, where that hydrodynamic-type
system describes classical mechanical systems with one-and-a-half degree of
freedom and with polynomial first integrals.}
\begin{equation}\label{uno}
\left\{
\begin{array}{l}
a_{y}=2a^{1/2}\left(  \overset{N-1}{\underset{m=1}{\sum}}b_{m}\right)
_{x}+a^{-1/2}a_{x}\left(  \overset{N-1}{\underset{m=1}{\sum}}b_{m}\right)
\\
\\
(b_{k})_{y}=a^{-1/2}b_{k}(b_{k})_{x}-[1+(b_{k})^{2}](a^{-1/2})_{x},\text{
\ }k=1,2,...,N-1\,.
\end{array}
\right.
\end{equation}
As we mentioned above, this system is integrable by the Generalized Hodograph Method (see \cite{Tsar,tsar91}). This means that any solution of this system determines an integrable geodesic flow (see details in \cite{MaxTsarMB}). We emphasize that description of integrable geodesic flows selected by Hamiltonian \eqref{eq.ham.we.considered} and by homogeneous polynomial first integral \eqref{Polynomial} reduces to the integrability of the hydrodynamic--type system \eqref{uno}. Moreover, once solutions of this system
are found, one can choose any conservation law density (see
\eqref{eq.p.q}, here instead of $p$ we use $h_k$ and instead of $q$ we use $b_k$)
\[
h_{k}=\frac{b_{k}}{\sqrt{1+(b_{k})^{2}}}
\]
to determine elements of the conservation law (for any index $k$, the density and the flux correspondingly)
\begin{equation*}
g^{12}=h_{k},\text{ \ }a_{N-1}=\sqrt{\frac{1-(h_{k})^{2}}{a}}
\end{equation*}
of the inverse reciprocal transformation \eqref{eq.recip.non.so}.
Since this transformation can be determined for any index $k$, we have $N-1$ different choices how to connect metrics \eqref{eq.metric.nonso} and \eqref{eq.metric.nonso.2}.
Thus any solution found in semi-geodesic coordinates can be recomputed back to
our original coordinates $(x^1,x^2)$. Here we briefly discuss the integrability of system \eqref{uno}. Its integrability is based on two important properties:
existence of Riemann invariants (in these coordinates any hydrodynamic--type system takes a diagonal form) and existence of infinitely many conservation laws. First we explain how to compute these conservation laws. By introducing the so called moments
$$
B^k=\frac{1}{k+1}\sum_{m=1}^N(b^m)^{k+1}\,,
$$
system \eqref{uno} assumes the form
\begin{equation*}
\left\{
\begin{array}{l}
a_{y}=2a^{1/2}
B^0_{x}+a^{-1/2}B^0a_{x}
\\
\\
B^0_y=a^{-1/2}B^1_x - (N+2B^1)(a^{-1/2})_x
\\
\\
B^k_y=a^{-1/2}B^{k+1}_x + (kB^{k-1}+(k+2)B^{k+1})(a^{-1/2})_x\,,\quad k=1,2,\dots
\end{array}
\right.
\end{equation*}
System \eqref{uno} has infinitely many polynomial conservation laws whose densities and fluxes depend only on $a$ and moments $B^k$. One can derive them either iteratively step by step or from the equation of the Riemann surface \eqref{lyambda} (see details below and \cite{classmech,MaxTsarMB}). For instance, first two conservation laws are
$$
a_{y}=\left(2a^{1/2}B^0\right)_x\,, \quad \left(B^0a^{3/2}\right)_y=\left( a\left( \frac{3}{2}(B^0)^2+B^1+\frac{N}{2} \right) \right)_x .
$$
Moreover, system \eqref{uno} is diagonalizable, i.e. it can be put in the form
\begin{equation}\label{Riemann}
r_{t}^{i}=\mu_{i}(\mathbf{r})r_{x}^{i}\,\quad i=1,\dots, N
\end{equation}
where $r^k$ are Riemann invariants, determined by the condition $\tilde{\lambda}_q=0$, i.e., they are
branch points of the Riemann surface determined by the equation \eqref{lyambda}. More precisely, $r^{k}(\mathbf{b})=\tilde{\lambda}(\mathbf{b},q)|_{q=q_{k}
(\mathbf{b})}$, where $\mathbf{b}=(b^1,\dots,b^{N-1})$, i.e., in our case,
the $N$ distinct roots $q_{k}(\mathbf{b})$ determined by the condition
\[
N\frac{q}{1+q^{2}}=\overset{N-1}{\underset{m=1}{\sum}}\frac{1}{q-b_{m}}
\]
and then substituted into the equation of Riemann surface \eqref{lyambda} give the Riemann invariants
\begin{equation*}
r^k(x,y)=a^{-1/2}(1+q_k^{2})^{-N/2}\overset{N-1}{\underset{m=1}{
{\displaystyle\prod}
}}(q_k-b_{m})\,.
\end{equation*}
So, if $q\rightarrow q_{k}(\mathbf{b})$, $\tilde{\lambda
}(\mathbf{b},q)\rightarrow r^{k}(\mathbf{b})$, $\tilde{\lambda}_q\rightarrow 0$, then \eqref{liouville} leads to \eqref{Riemann},
with characteristic velocities $\mu_{k}(\mathbf{r}(a,\mathbf{b}))=a^{-1/2}q_{k}(\mathbf{b})$.
Hydrodynamic--type systems which are simultaneously diagonalizable and possess infinitely many conservation laws are integrable by the Generalized Hodograph Method (see \cite{Tsar,tsar91}). For instance, if $N=2$, then characteristic velocities are
\[
\mu_{1}=a^{-1/2}\left(  b_{1}+\sqrt{(b_{1})^{2}+1}\right)  ,\text{ \ }\mu
_{2}=a^{-1/2}\left(  b_{1}-\sqrt{(b_{1})^{2}+1}\right)  ;
\]
and the corresponding Riemann invariant are
\[
r^{1}=\frac{1}{2}\frac{a^{-1/2}}{b_{1}+\sqrt{(b_{1})^{2}+1}},\text{ \ \ }%
r^{2}=\frac{1}{2}\frac{a^{-1/2}}{b_{1}-\sqrt{(b_{1})^{2}+1}}.
\]
Thus, hydrodynamic-type system \eqref{Riemann} becomes
\[
r_{t}^{1}=-2r^{2}r_{x}^{1},\text{ \ }r_{t}^{2}=-2r^{1}r_{x}^{2}.
\]
This system (as well as the corresponding metric written in semi-geodesic
coordinates) is discussed in detail in \cite{MaxTsarMB}.

\begin{remark}
The first integral investigated in \cite{Bialy.Mironov} and \cite{MaxTsarMB}
is different (cf. \eqref{polynom}):
\[
f=(\tilde{p}_{1})^{N}+\tilde{p}_{2}(\tilde{p}_{1})^{N-1}+\tilde{a}_{1}
(\tilde{p}_{2})^{2}(\tilde{p}_{1})^{N-2}+\tilde{a}_{2}(\tilde{p}_{2})^{3}(\tilde{p}_{1})^{N-3}+...+\tilde{a}_{N-2}(\tilde{p}_{2})^{N-1}\tilde
{p}_{1}+\tilde{a}_{N-1}(\tilde{p}_{2})^{N}.
\]
However the integration procedure based on the Generalized Hodograph Method
(see more details in \cite{Tsar,tsar91}) is precisely the same as in
\cite{MaxTsarMB}.
\end{remark}

\section{Conclusion}

In this paper we presented an approach, described in details in Section \ref{sec.1}, that allows to reduce the problem of the description of integrable geodesic flows selected by homogeneous polynomial first integrals to the integrability of semi--Hamiltonian hydrodynamic--type systems possessing a variety of inequivalent reductions. In fact, the polynomial ansatz \eqref{Polynomial} for the first integral depends on two natural numbers: one of them is the degree of the polynomial and the other is the number of its non-zero coefficients.
Moreover, the considered hydrodynamic--type systems can be written in evolutionary form
by an appropriate reciprocal transformation to semi-geodesic coordinates. All of them can be integrated by the Generalized Hodograph Method. However, even in the two components cases, such solutions can be presented just in implicit form. Nevertheless, this two component case is very interesting since characteristic velocities can be expressed explicitly via Riemann invariants. This particular research will be the topic of a future investigation.

\section*{Acknowledgements}

The authors thank Andrey Mironov and Sergey Tsarev for important discussions.
Both authors were partially supported by the grant ``Finanziamento giovani
studiosi - Metriche proiettivamente equivalenti, equazioni di Monge--Amp\`ere
e sistemi integrabili'', University of Padova 2013-2015 and by the project ``FIR (Futuro in Ricerca) 2013 -
Geometria delle equazioni differenziali''. The second author was partially
supported by the grant of Presidium of RAS \textquotedblleft Fundamental
Problems of Nonlinear Dynamics\textquotedblright\ and by the RFBR grant
15-01-01671-a. The first author is member of G.N.S.A.G.A. of I.N.d.A.M.

\bibliographystyle{elsarticle-num}
\bibliography{GianniBib}

\begin{thebibliography}{10}
\expandafter\ifx\csname url\endcsname\relax
  \def\url#1{\texttt{#1}}\fi
\expandafter\ifx\csname urlprefix\endcsname\relax\def\urlprefix{URL }\fi
\expandafter\ifx\csname href\endcsname\relax
  \def\href#1#2{#2} \def\path#1{#1}\fi

\bibitem{Bialy.Mironov}
M.~Bialy, A.~E. Mironov,
  \href{http://dx.doi.org/10.1016/j.geomphys.2014.08.006}{Integrable geodesic
  flows on 2-torus: formal solutions and variational principle}, J. Geom. Phys.
  87 (2015) 39--47.
\newblock \href {http://dx.doi.org/10.1016/j.geomphys.2014.08.006}
  {\path{doi:10.1016/j.geomphys.2014.08.006}}.
\newline\urlprefix\url{http://dx.doi.org/10.1016/j.geomphys.2014.08.006}

\bibitem{Da}
G.~Darboux, Le\c cons sur la th\'eorie g\'en\'erale des surfaces. {III}, {IV},
  Les Grands Classiques Gauthier-Villars. [Gauthier-Villars Great Classics],
  \'Editions Jacques Gabay, Sceaux, 1993, lignes g{\'e}od{\'e}siques et
  courbure g{\'e}od{\'e}sique. Param{\`e}tres diff{\'e}rentiels.
  D{\'e}formation des surfaces. [Geodesic lines and geodesic curvature.
  Differential parameters. Deformation of surfaces], D{\'e}formation infiniment
  petite et repr{\'e}sentation sph{\'e}rique. [Infinitely small deformation and
  spherical representation], Reprint of the 1894 original (III) and the 1896
  original (IV), Cours de G{\'e}om{\'e}trie de la Facult{\'e} des Sciences.
  [Course on Geometry of the Faculty of Science].

\bibitem{Krug}
B.~Kruglikov, \href{http://dx.doi.org/10.1016/j.geomphys.2008.03.005}{Invariant
  characterization of {L}iouville metrics and polynomial integrals}, J. Geom.
  Phys. 58~(8) (2008) 979--995.
\newblock \href {http://dx.doi.org/10.1016/j.geomphys.2008.03.005}
  {\path{doi:10.1016/j.geomphys.2008.03.005}}.
\newline\urlprefix\url{http://dx.doi.org/10.1016/j.geomphys.2008.03.005}

\bibitem{vagner}
V.~Vagner, On the problem of determining the invariant characteristics of
  {L}iouville surfaces, Abh. Sem. Vektor- und Tensoranalysis [Trudy Sem.
  Vektor. Tenzor. Analizu] 5 (1941) 246--249.

\bibitem{BMM}
R.~L. Bryant, G.~Manno, V.~S. Matveev,
  \href{http://dx.doi.org/10.1007/s00208-007-0158-3}{A solution of a problem of
  {S}ophus {L}ie: normal forms of two-dimensional metrics admitting two
  projective vector fields}, Math. Ann. 340~(2) (2008) 437--463.
\newblock \href {http://dx.doi.org/10.1007/s00208-007-0158-3}
  {\path{doi:10.1007/s00208-007-0158-3}}.
\newline\urlprefix\url{http://dx.doi.org/10.1007/s00208-007-0158-3}

\bibitem{KKM}
E.~G. Kalnins, J.~M. Kress, W.~Miller, Jr.,
  \href{http://dx.doi.org/10.1063/1.1894985}{Second order superintegrable
  systems in conformally flat spaces. {II}. {T}he classical two-dimensional
  {S}t\"ackel transform}, J. Math. Phys. 46~(5) (2005) 053510, 15.
\newblock \href {http://dx.doi.org/10.1063/1.1894985}
  {\path{doi:10.1063/1.1894985}}.
\newline\urlprefix\url{http://dx.doi.org/10.1063/1.1894985}

\bibitem{inv}
V.~S. Matveev, V.~V. Shevchishin,
  \href{http://dx.doi.org/10.1016/j.geomphys.2010.02.002}{Differential
  invariants for cubic integrals of geodesic flows on surfaces}, J. Geom. Phys.
  60~(6-8) (2010) 833--856.
\newblock \href {http://dx.doi.org/10.1016/j.geomphys.2010.02.002}
  {\path{doi:10.1016/j.geomphys.2010.02.002}}.
\newline\urlprefix\url{http://dx.doi.org/10.1016/j.geomphys.2010.02.002}

\bibitem{super}
V.~S. Matveev, V.~V. Shevchishin,
  \href{http://dx.doi.org/10.1016/j.geomphys.2011.02.012}{Two-dimensional
  superintegrable metrics with one linear and one cubic integral}, J. Geom.
  Phys. 61~(8) (2011) 1353--1377.
\newblock \href {http://dx.doi.org/10.1016/j.geomphys.2011.02.012}
  {\path{doi:10.1016/j.geomphys.2011.02.012}}.
\newline\urlprefix\url{http://dx.doi.org/10.1016/j.geomphys.2011.02.012}

\bibitem{Bialy}
M.~Bialy, \href{http://dx.doi.org/10.1007/s00039-010-0069-4}{Integrable
  geodesic flows on surfaces}, Geom. Funct. Anal. 20~(2) (2010) 357--367.
\newblock \href {http://dx.doi.org/10.1007/s00039-010-0069-4}
  {\path{doi:10.1007/s00039-010-0069-4}}.
\newline\urlprefix\url{http://dx.doi.org/10.1007/s00039-010-0069-4}

\bibitem{classmech}
M.~V. Pavlov, S.~P. Tsarev,
  \href{http://dx.doi.org/10.1090/trans2/234/17}{Classical mechanical systems
  with one-and-a-half degrees of freedom and {V}lasov kinetic equation}, in:
  Topology, geometry, integrable systems, and mathematical physics, Vol. 234 of
  Amer. Math. Soc. Transl. Ser. 2, Amer. Math. Soc., Providence, RI, 2014, pp.
  337--371.
\newblock \href {http://dx.doi.org/10.1090/trans2/234/17}
  {\path{doi:10.1090/trans2/234/17}}.
\newline\urlprefix\url{http://dx.doi.org/10.1090/trans2/234/17}

\bibitem{MaxTsarMB}
M.~V. Pavlov, S.~P. Tsarev,
  \href{http://dx.doi.org/10.1088/1751-8113/49/17/175201}{On local description
  of two-dimensional geodesic flows with a polynomial first integral}, J. Phys.
  A 49~(17) (2016) 175201, 20.
\newblock \href {http://dx.doi.org/10.1088/1751-8113/49/17/175201}
  {\path{doi:10.1088/1751-8113/49/17/175201}}.
\newline\urlprefix\url{http://dx.doi.org/10.1088/1751-8113/49/17/175201}

\bibitem{OdPavSok}
A.~V. Odesski{\u\i}, M.~V. Pavlov, V.~V. Sokolov,
  \href{http://dx.doi.org/10.1007/s11232-008-0020-0}{Classification of
  integrable {V}lasov-type equations}, Teoret. Mat. Fiz. 154~(2) (2008)
  249--260.
\newblock \href {http://dx.doi.org/10.1007/s11232-008-0020-0}
  {\path{doi:10.1007/s11232-008-0020-0}}.
\newline\urlprefix\url{http://dx.doi.org/10.1007/s11232-008-0020-0}

\bibitem{maksgen}
M.~V. Pavlov,
  \href{http://dx.doi.org/10.1088/0305-4470/39/34/014}{Classification of
  integrable hydrodynamic chains and generating functions of conservation
  laws}, J. Phys. A 39~(34) (2006) 10803--10819.
\newblock \href {http://dx.doi.org/10.1088/0305-4470/39/34/014}
  {\path{doi:10.1088/0305-4470/39/34/014}}.
\newline\urlprefix\url{http://dx.doi.org/10.1088/0305-4470/39/34/014}

\bibitem{algebra}
M.~V. Pavlov,
  \href{http://dx.doi.org/10.1007/s00220-007-0235-1}{Algebro-geometric approach
  in the theory of integrable hydrodynamic type systems}, Comm. Math. Phys.
  272~(2) (2007) 469--505.
\newblock \href {http://dx.doi.org/10.1007/s00220-007-0235-1}
  {\path{doi:10.1007/s00220-007-0235-1}}.
\newline\urlprefix\url{http://dx.doi.org/10.1007/s00220-007-0235-1}

\bibitem{Tsar}
S.~P. Tsar{\"e}v, Poisson brackets and one-dimensional {H}amiltonian systems of
  hydrodynamic type, Dokl. Akad. Nauk SSSR 282~(3) (1985) 534--537.

\bibitem{tsar91}
S.~P. Tsar{\"e}v, The geometry of {H}amiltonian systems of hydrodynamic type.
  {T}he generalized hodograph method, Izv. Akad. Nauk SSSR Ser. Mat. 54~(5)
  (1990) 1048--1068.

\bibitem{GibbonsTsarev}
J.~Gibbons, S.~P. Tsarev,
  \href{http://dx.doi.org/10.1016/0375-9601(95)00954-X}{Reductions of the
  {B}enney equations}, Phys. Lett. A 211~(1) (1996) 19--24.
\newblock \href {http://dx.doi.org/10.1016/0375-9601(95)00954-X}
  {\path{doi:10.1016/0375-9601(95)00954-X}}.
\newline\urlprefix\url{http://dx.doi.org/10.1016/0375-9601(95)00954-X}

\bibitem{kozlov}
V.~V. Kozlov, \href{http://dx.doi.org/10.1007/BF01158890}{Polynomial integrals
  of dynamical systems with one-and-a-half degrees of freedom}, Mat. Zametki
  45~(4) (1989) 46--52, 125.
\newblock \href {http://dx.doi.org/10.1007/BF01158890}
  {\path{doi:10.1007/BF01158890}}.
\newline\urlprefix\url{http://dx.doi.org/10.1007/BF01158890}

\end{thebibliography}

\end{document}